\documentclass[12pt]{article}
\usepackage{latexsym,amssymb,upref,amsmath,amsthm, amsfonts,authblk}
\usepackage{color}
\usepackage{graphicx}
\usepackage[margin=1in]{geometry}
\usepackage{hyperref}

\usepackage{float}

\newtheorem{thm}{Theorem}

\newtheorem{claim}[thm]{Claim}



\author{J\'ozsef Balogh\thanks{Department of Mathematics, University of Illinois at Urbana-Champaign, IL, 
USA. Email: \texttt{jobal@illinois.edu}.
Partially supported by NSF Grant DMS-1764123 and Arnold O. Beckman Research Award (UIUC) Campus Research Board 18132, Simons Fellowship 
and the Langan Scholar Fund (UIUC).}
\qquad
John Finlay\thanks{Department of Mathematical Sciences, University of Montana. Email: \texttt{john.finlay@umontana.edu}.}
\qquad
Cory Palmer\thanks{Department of Mathematical Sciences, University of Montana. Email: \texttt{cory.palmer@umontana.edu}.
Research supported by a grant from the Simons Foundation \#712036.}}

\title{Rainbow connectivity of randomly perturbed graphs}

\begin{document}

\maketitle

\begin{abstract}
In this note we examine the following random graph model: for an arbitrary graph $H$, with quadratic many edges, construct a graph $G$ by randomly adding $m$ edges to $H$ and randomly coloring the edges of $G$ with $r$ colors. We show that for $m$ a large enough constant and $r \geq 5$, every pair of vertices in $G$ are joined by a rainbow path, i.e., $G$ is {\it rainbow connected}, with high probability. This confirms a conjecture of Anastos and Frieze [{\it J. Graph Theory} {\bf 92} (2019)] who proved the statement for $r \geq 7$ and resolved the case when $r \leq 4$ and $m$ is a function of $n$.
\end{abstract}

\section{Introduction}

The {\it randomly perturbed graph model} is the following: for a fixed positive constant $\delta>0$, let $\mathcal{G}(n, \delta)$ be the set of graphs on vertex set $[n]$ with minimum degree at least $\delta n $. A graph $H=(V, E) $ is chosen arbitrarily from $\mathcal{G}(n,\delta)$ and a set $R$ of $m$ edges are chosen uniformly at random from $\binom{[n]}{2} \setminus E$ (i.e., those edges not in $H$) and added to $H$. The resulting {\it perturbed graph} is 
\[
G_{H, m}=(V, E \cup R).
\]
This model was introduced by Bohman, Frieze, and Martin \cite{BFM}. As taking $\delta=0$ gives the standard random graph, this model can be viewed as a generalization of the classic Erd\H os-R\'enyi random graph model.
Given a class of graphs $\mathcal{G}$ and a monotone increasing property $P$, a natural line of questioning is: for a graph $H\in \mathcal{G}$ can we perturb the graph, specifically by adding edges randomly so that the resulting graph satisfies $P$ with high probability?
For example, in \cite{BFM} it is shown that for an $n$-vertex dense graph $H$, the perturbed graph $G_{H,m}$ is Hamiltonian with high probability (w.h.p.) if $m$ is at least linear in $n$ and that there are dense non-Hamiltonian graphs $H$ such that it is necessary for $m$ to be linear in $n$ for $G_{H,m}$ to be Hamiltonian w.h.p.

Various other properties of randomly perturbed graphs have been investigated including clique number, chromatic number, diameter and connectivity \cite{BFKM}, spanning trees \cite{BHKMPP, BMPP, JK, KKS, KRS}, Ramsey properties \cite{DT, DMT}, tilings \cite{BTW}, and containing some fixed subgraph \cite{MST}.

Anastos and Frieze \cite{AF} examined random edge-colorings of perturbed graphs. Let $G^r_{H, m}$ be the graph $G_{H,m}$ equipped with an $r$-edge-coloring where the colors of the edges are independently and uniformly selected from $[r]$ (in this way the coloring need not be proper). In \cite{AF} it was shown that a linear number of edges $m$ and number of colors $r \geq (120 - 20 \log \delta)n$ suffice for $G^r_{H,m}$ to have a rainbow-Hamilton cycle w.h.p. Later Aigner-Horev and Hefetz \cite{AH} improved the value of $r$ to the asymptotically best-possible $(1+o(1))n$. Very recently Katsamaktsis and Letzter \cite{KL} showed that the optimal number of colors $n$ suffices.

An edge-colored graph is {\it rainbow connected} if there is a rainbow path (i.e., each edge has a distinct color) between every pair of vertices. 
Rainbow connectivity has been studied in graphs in general as well as in the standard Erd\H os-R\'enyi random graph model. Typically, the question is to find the minimum $r$ such that a graph $G$ has an $r$-edge-coloring such that $G$ is rainbow connected. In this case we say that $G$ has {\it rainbow connection number} $r$. This parameter was introduced by Chartrand, Johns, McKeon and Zhang \cite{CJMZ}. For a survey of rainbow connectivity in graphs, see Li and Sun \cite{LS}.

In random graphs the goal is to determine a threshold on the edge probability $p$ such that w.h.p.\ $G(n,p)$ has rainbow connection number $r$. Caro, Lev, Roditty, Tuza and Yuster~\cite{CLRTY} showed that $p = \sqrt{\frac{\log n}{n}}$ is a sharp threshold for rainbow connection number $r=2$ and He and Liang~\cite{HL} proved that $\frac{(\log n)^{1/r}}{n^{1-1/r}}$ is a sharp threshold for $r\geq 3$. A further refinement for $r\geq 3$ was given by Heckel and Riordan~\cite{HR}.

In a perturbed graph $G_{H, m}^r$,
Anastos and Frieze \cite{AF} proved that if $r\geq 7 $ and $m=\omega (1)$, then with high probability $G_{H, m}^r$ is rainbow connected. 
Moreover, they showed the following.

\begin{thm}[Anastos, Frieze \cite{AF}]\label{AF-main}
\begin{enumerate}
    \item[(i)] For every $\delta>0$, if $m \geq 60 \delta^{-2} \log n$, then w.h.p.\ $G_{H,m}^3$ is rainbow connected.
    
    \item[(ii)] For every $0.1 \geq \delta>0$, there exists $H \in \mathcal{G}(n,\delta)$ such that if $m \leq 0.5 \log n$, then w.h.p.\ $G_{H,m}^4$ is not rainbow connected.
    
    \item[(iii)] If $r \geq 7$ and $m = \omega(1)$, then w.h.p.\ $G_{H,m}^r$ is rainbow connected.
\end{enumerate}
\end{thm}

It is not immediately clear that rainbow connectivity is monotone in $r$ for constant $m$.
For $r=3$ colors, (i) shows that if $m$ has order of magnitude $\log n$, then $G^3_{H,m}$ is rainbow connected w.h.p., which shows that (ii) is best possible.
Anastos and Frieze \cite{AF} put forth that likely part (iii) holds for $5$ or $6$. In this note we prove the result for $r = 5$ and $r=6$. Note that part (ii) implies that this cannot be improved to $r \geq 4$.

\begin{thm}\label{main-thm}
For every $\delta >0$,
if $r \geq 5$ and  $m \geq {4}\delta^{-2} \log {2500}\delta^{-2}$, then w.h.p.\ $G_{H,m}^r$ is rainbow connected.
\end{thm}

We make no particular attempt to optimize the constant $m$ in Theorem~\ref{main-thm}, however proving sharp bounds remains an interesting problem. A natural modification of this model would be to introduce some deterministic part to the edge-coloring which may lead to different values of $m$ and $r$.

\section{Proof of Theorem~\ref{main-thm}}
Fix $r \geq 5$
and $0<\delta<1$ and integers $m \geq 4\delta^{-2} \log 2500\delta^{-2}$ and $k$ such that $6 \leq \delta k \leq 7$. Let $H$ be an arbitrary $n$-vertex graph with minimum degree at least $\delta n$. Randomly add $m$ edges from $\binom{[n]}{2}$ to $H$ (ignoring duplicate edges) and randomly $5$-color the edges of the resulting graph $G_{H,m}^r$. Note that this is a slightly weaker assumption than adding $m$ edges from $\binom{[n]}{2} \setminus E(H)$ to $H$, which will simplify arguments later.
At this point, we have two sources of randomness: the $m$ edges added to $H$ and the edge-coloring of $G_{H,m}^r$.

We begin by building $t = 100 \log n$ vertex sets of size $k$ in $H$ such that every pair of vertices $u,v$ are in the neighborhoods of a large proportion of these sets.

\begin{claim}\label{index-claim}
There exists vertex sets $S_1,S_2,\dots, S_t$, each of $k$ vertices in $V(H)$ such that for every pair of vertices $u,v \in V(H)$  there is an index set $I_{u,v} \subset [t]$ such that $|I_{u,v}| \geq 0.6t$ and for all $i \in I_{u,v}$ we have $u,v \in N(S_i)$.
\end{claim}

\begin{proof}
Select uniformly at random with replacement $t$ sets $S_1,S_2,\dots, S_t$ each of $k$ vertices from $V(H)$.

For distinct vertices $u,v$ the probability $\mathbb{P}(u \not \in N(S_i) \textrm{  or } v \not \in N(S_i))$ is at most
\begin{align*}
  & \mathbb{P}(u \not \in N(S_i)) + \mathbb{P}(v \not \in N(S_i)) 
 \leq  2 \mathbb{P}(u \not \in N(S_i)) 
  = 2\mathbb{P}(S_i \cap N(u) = \emptyset) \\
  & \leq  2\frac{\binom{(1-\delta)n}{k}}{\binom{n}{k}} 
 \leq 2(1-\delta)^k < 2 \exp(-\delta k) < 0.01,
\end{align*}
as $k \geq 6\delta^{-1}$.
Thus, $\mathbb{P}(u,v\in N(S_i)) \geq 0.99$.

Let $X_i$ be the indicator random variable for the event that $u,v \in N(S_i)$. Now $X = X_1+ X_2 + \dots + X_t$ 
is the number of sets $S_i$ such that $u,v \in N(S_i)$.
Therefore, $\mathbb{E}[X] \geq 0.99t$.

As each set $S_i$ is selected uniformly at random with replacement, the random variables $X_i$ are independent. 
Therefore, a Chernoff bound (see \cite[pg. 66]{textbook}) gives
\begin{align*}
 \mathbb{P}(X \leq (1-\alpha) \mathbb{E}[X]) \leq \exp\left(-\frac{1}{2} \alpha^2 \mathbb{E}[X]\right) 
\end{align*}
for $0 < \alpha < 1$.
With $\alpha =  0.39$ this gives 
\[
\mathbb{P}(X \leq 0.6 t) \leq \mathbb{P}(X \leq 0.61 \cdot 0.99 t) \leq  \exp\left(-\frac{1}{2} (0.39)^2 \cdot 0.99 \cdot 100 \log n\right) < n^{-7} <  n^{-2}.
\]
                                    
Therefore, by the union bound, the probability that some $u,v$ is not contained in $0.6t$ sets $S_i$ is less than $1$. Thus, with positive probability the desired vertex sets exists which completes the proof.
\end{proof}

Put $S = \bigcup S_i$.
A set $S_i$ is {\it good} if for every pair $a,b \in S_i$ there exists an edge of $G_{H,m}^r$ between $N(a) \setminus  S$ and $N(b) \setminus S$.  Note that these two sets are not necessarily disjoint, so such an edge may be contained in their intersection.

\begin{claim}\label{good-claim}
$\mathbb{P}(S_i \textrm{ is good}) > 0.99$ for $n$ large enough.
\end{claim}

\begin{proof}
Here we only use the randomness of the $m$ edges added to $H$.
Fix $a,b \in S_i$.
As $|S| = |\cup S_i| \leq kt = 100 k \log n$,  we can choose $n$ large enough such that
 $N(a)\setminus S$ and $N(b)\setminus S$ both have size at least $\frac{1}{2} \delta n$.
Note that $N(a)\setminus S$ and $N(b) \setminus S$ are not necessarily disjoint,
so the probability that there is no edge between $N(a)\setminus S$ and $N(b)\setminus S$ is at most
\[
\left(\frac{\binom{n}{2} - \binom{\frac{1}{2} \delta n}{2}}{\binom{n}{2}}\right)^m \leq \left(1 - \frac{1}{4}\delta^2 \right)^m < \exp\left(-\frac{1}{4} \delta^2m\right).
\]
The set $S_i$ has $k$ vertices, so by the union bound, the probability that $S_i$ is not good is less than
\[
 \binom{k}{2} \exp\left(-\frac{1}{4} \delta^2m\right) \leq \frac{1}{2} \left( \frac{7}{\delta} \right)^2 \exp\left(-\frac{1}{4} \delta^2m\right)<0.01,
\]
as $m \geq 4\delta^{-2} \log 2500\delta^{-2}$.
\end{proof}

Fix arbitrary vertices $u$ and $v$ in $G_{H,m}^r$. We estimate the probability that there is a rainbow $u$--$v$ path (i.e.\ a path with end-vertices $u$ and $v$) of length at most $5$. Here we only use the randomness of the edge-coloring of $G_{H,m}^r$. If $uv$ is an edge, then we immediately have a rainbow $u$--$v$ path, so assume that $uv$ is not an edge.

Let $I_{u,v}$ be the index set guaranteed by Claim~\ref{index-claim}. For each $i \in I_{u,v}$, let us estimate the probability that there is a rainbow $u$--$v$ path of length at most $5$ using vertices in $S_i$. Let $a \in S_i$ be a neighbor of $u$ and let $b \in S_i$ be a neighbor of $v$. If $a=b$, then we have a $u$--$v$ path of length $2$ which is rainbow with probability
$(r-1)/r \geq 4/5$, so assume that $a \neq b$.

\begin{figure}[H]
\begin{center}
{\includegraphics{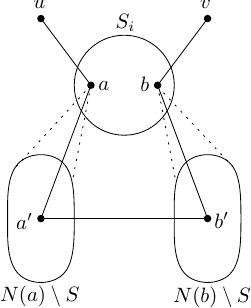}}
\end{center}
\caption{A $u$--$v$ path of length $5$.}
\end{figure}

By Claim~\ref{good-claim}, with probability at least $0.99$, there is an edge $a'b'$ between $N(a)\setminus S$ and $N(b)\setminus S$. 
Conditioning on the existence of $a'b'$, the probability that the path $uaa'b'bv$ is rainbow is
$\frac{(r-1)(r-2)(r-3)(r-4)}{r^5} \geq \frac{4!}{5^4}$.
Thus, the 
probability that there is a rainbow $u$--$v$ path using vertices $u,v,a,b$ is at least
$0.99 \cdot \frac{4!}{5^4}$. (Here we needed $r\geq 5$ colors as with fewer colors this path cannot be rainbow.)

Therefore, the probability that there is no rainbow $u$--$v$ path of length at most $5$ is at most
\begin{align*}
\left(1-0.99 \cdot\frac{4!}{5^4}\right)^{|I_{uv}|} & \leq \left(1-0.99 \cdot\frac{4!}{5^4}\right)^{0.6t} \leq  \exp\left(-0.99 \cdot \frac{4!}{5^4} \cdot 0.6 t\right) \\
&= \exp\left(-0.99 \cdot \frac{4!}{5^4} \cdot 0.6 \cdot 100 \log n\right) < n^{-2.25}
 = o(n^{-2}).
\end{align*}

Now, the union bound implies that the probability that there is a pair $u,v$ not connected by a rainbow path of length at most $5$ is $o(1)$, i.e., w.h.p.\ $G_{H,m}^r$ is rainbow connected. \hfill \qedsymbol

\section*{Acknowledgements}
We thank the anonymous referees for their careful reading of the manuscript and for many helpful comments.

\end{document}